\documentclass[11pt,a4paper]{article}
\usepackage[margin=25mm,headheight=15pt]{geometry}
\usepackage{amsmath,amssymb,amsthm,mathtools,booktabs,longtable,array}
\usepackage{fontspec,xcolor,hyperref,fancyhdr,listings}
\hypersetup{colorlinks=true,linkcolor=blue!45!black,citecolor=blue!45!black,urlcolor=blue!45!black}
\allowdisplaybreaks[2]
\newtheorem{theorem}{Theorem}[section]
\newtheorem{lemma}[theorem]{Lemma}
\newtheorem{proposition}[theorem]{Proposition}
\theoremstyle{remark}\newtheorem{remark}[theorem]{Remark}
\newcommand{\R}{\mathbb R}\newcommand{\spn}{\operatorname{span}}
\newcommand{\E}{\widetilde e}\newcommand{\Ric}{\widetilde\rho^B}
\newcommand{\LB}{\mathcal L^B}\newcommand{\diag}{\operatorname{diag}}
\title{A complete parameter classification of left-invariant Bott affine Ricci solitons for two rotated distributions on three-dimensional Lorentzian Lie groups}
\author{FANQI LIN}\date{}
\begin{document}\maketitle
\begin{abstract}
We classify all left-invariant affine Ricci solitons associated with the Bott connections of two fixed rank-two distributions on the seven standard families of three-dimensional Lorentzian Lie groups. The distributions are defined by a fixed orthogonal change of the spacelike frame. A universal adapted-frame calculation provides explicit connection matrices, curvature operators and the symmetrized Ricci tensor. The equation in the orthogonal-complement direction forces every left-invariant solution to be steady. The remaining equations yield elementary parameter families and finitely many real algebraic branches: four normalized branches for $G_2$, six branches for each sign parameter in $G_4$, and two quartic branches in $G_7$ in addition to its elementary families. Resultants, subresultants, exact polynomial reduction and Sturm root isolation establish necessity, sufficiency and exhaustiveness. Degenerate cases are treated separately, including all families with nonunique soliton vector fields. An invertible frame transformation relates the two distribution problems. The classification is in the stated structure parameters; no quotient by isometries, automorphisms or scaling is taken.
\end{abstract}
\noindent\textbf{Keywords:} Lorentzian Lie group; Bott connection; affine Ricci soliton; real algebraic classification; left-invariant geometry.

\noindent\textbf{Mathematics Subject Classification (2020):} 53C50, 53C25, 22E25.
\tableofcontents
\section{Introduction}
Left-invariant geometry turns tensor equations on a Lie group into finite algebraic systems on its Lie algebra. In Lorentzian signature, this reduction remains effective, but signs in the Koszul formula and Ricci contraction require particular care. If the connection depends on a chosen distribution, that distribution is part of the data and must remain visible in the classification.

Wu and Wang studied Bott connections and affine Ricci solitons for three distributions on three-dimensional Lorentzian Lie groups \cite{WuWang}. Homogeneous Lorentzian geometry provides the broader setting \cite{Calvaruso}. Batat and Onda classified algebraic Ricci solitons on the standard families \cite{BatatOnda}; their equation is different from the affine equation considered here. We use these works for the geometric framework and standard models, without transferring results between distinct soliton notions.

Fix a pseudo-orthonormal left-invariant frame $\{e_1,e_2,e_3\}$ with $e_3$ timelike, and set
\begin{equation}\label{eq:rotation}
\E_1=(e_1+e_2)/\sqrt2,\quad\E_2=(e_1-e_2)/\sqrt2,\quad\E_3=e_3.
\end{equation}
We investigate $D_1=\spn\{\E_1,\E_3\}$ and $D_2=\spn\{\E_2,\E_3\}$. For a fixed bracket, replacing the distribution generally changes the Bott connection. Consequently, a transfer between the two problems must include the induced changes in bracket parameters and vector coordinates.

\subsection{Scope and principal results}
We study Bott affine Ricci solitons on three-dimensional Lorentzian Lie groups equipped with a fixed left-invariant Lorentzian metric and a specified distribution. A quadruple \((G,g,D,V)\), where \(G\) is a three-dimensional Lie group, \(g\) is a left-invariant Lorentzian metric on \(G\), \(D\) is one of the specified distributions, and \(V\) is a left-invariant vector field, is called a Bott affine Ricci soliton if there exists a constant \(\mu\in\R\) such that
\begin{equation}\label{eq:soliton}
	\LB_Vg+2\Ric+2\mu g=0.
\end{equation}
We classify all such solitons for the two fixed distributions under consideration. The classification exhausts all real solutions in the displayed structure parameters.
It does not take a quotient by isometries, automorphisms or scale, and it does not assert geodesic completeness.

\begin{theorem}[Overview]\label{thm:overview}
Every left-invariant solution for $G_1,\ldots,G_7$ and either distribution is steady: $\mu=0$. The complete solution sets are given in Sections~\ref{sec:unimodular} and~\ref{sec:nonunimodular}. The family $G_1$ has no solutions. The families $G_3,G_5,G_6$ have elementary parameter families. The normalized $G_2$ problem has four real algebraic branches. For each $\eta=\pm1$, $G_4$ has six branches. The $G_7$ classification contains elementary families and two additional real quartic branches.
\end{theorem}

The proof separates geometric reduction from real algebraic elimination. A common adapted-frame calculation gives five steady equations. All zero-parameter cases are treated before division. Algebraic candidates are then lifted to the original system and certified by exact remainders, while Sturm sequences count and isolate all real roots. The appendices retain both distributions' component systems and a reproducible verification procedure.

\section{Conventions and algebraic preliminaries}
\subsection{Metric, connection and Ricci contraction}
Write $g(e_i,e_j)=\varepsilon_i\delta_{ij}$, with $(\varepsilon_1,\varepsilon_2,\varepsilon_3)=(1,1,-1)$. Both distributions are nondegenerate and admit the orthogonal splitting $TG=D\oplus D^\perp$. We define
\begin{equation}\label{eq:bott}
\nabla^B_XY=\begin{cases}
\pi_D(\nabla^L_XY),&X,Y\in\Gamma(D),\\
\pi_D[X,Y],&X\in\Gamma(D^\perp),\ Y\in\Gamma(D),\\
\pi_{D^\perp}[X,Y],&X\in\Gamma(D),\ Y\in\Gamma(D^\perp),\\
\pi_{D^\perp}(\nabla^L_XY),&X,Y\in\Gamma(D^\perp),
\end{cases}
\end{equation}
where $\nabla^L$ is the Levi--Civita connection. No integrability assumption is imposed. The term Bott connection refers throughout to this four-block formula.

The mixed formulas satisfy the connection axioms. For $X\in D^\perp$, $Y\in D$, one has $\pi_D[X,fY]=X(f)Y+f\pi_D[X,Y]$ and $\pi_D[fX,Y]=f\pi_D[X,Y]$, because the additional term $-Y(f)X$ is removed by projection. The other mixed block is analogous; the equal-distribution blocks inherit the connection properties from $\nabla^L$.

Our sign conventions are
\begin{align}
R^B(X,Y)Z&=\nabla^B_X\nabla^B_YZ-\nabla^B_Y\nabla^B_XZ-\nabla^B_{[X,Y]}Z,\\
\rho^B(X,Y)&=-\sum_{i=1}^3\varepsilon_i g(R^B(X,\E_i)Y,\E_i),\\
\Ric(X,Y)&=\tfrac12\bigl(\rho^B(X,Y)+\rho^B(Y,X)\bigr),\\
(\LB_Vg)(X,Y)&=g(\nabla^B_XV,Y)+g(X,\nabla^B_YV).
\end{align}
In particular, $\rho^B$ is not assumed symmetric, and $\LB_Vg$ is not generally the ordinary Lie derivative. To simplify coefficients, put
\begin{equation}\label{eq:xyz}
V=\sqrt2x\E_1+\sqrt2y\E_2+z\E_3.
\end{equation}
All classification triples are $(x,y,z)$, not the frame coefficients $(\sqrt2x,\sqrt2y,z)$.

\subsection{Standard bracket families}
The models used below are the standard forms in three-dimensional Lorentzian Lie-group geometry \cite{Calvaruso,BatatOnda}. All parameters are real and satisfy the displayed restrictions:
\begin{align*}
G_1:\quad &[e_1,e_2]=\alpha e_1-\beta e_3,\quad
[e_1,e_3]=-\alpha e_1-\beta e_2,\\
&[e_2,e_3]=\beta e_1+\alpha e_2+\alpha e_3,\qquad\alpha\ne0;\\[2pt]
G_2:\quad &[e_1,e_2]=\gamma e_2-\beta e_3,\quad
[e_1,e_3]=-\beta e_2-\gamma e_3,\\
&[e_2,e_3]=\alpha e_1,\qquad\gamma\ne0;\\[2pt]
G_3:\quad &[e_1,e_2]=-\gamma e_3,\quad[e_1,e_3]=-\beta e_2,\quad
[e_2,e_3]=\alpha e_1;\\[2pt]
G_4:\quad &[e_1,e_2]=-e_2+(2\eta-\beta)e_3,\quad
[e_1,e_3]=-\beta e_2+e_3,\\
&[e_2,e_3]=\alpha e_1,\qquad\eta=\pm1;\\[2pt]
G_5:\quad &[e_1,e_2]=0,\quad[e_1,e_3]=\alpha e_1+\beta e_2,\quad
[e_2,e_3]=\gamma e_1+\delta e_2,\\
&\alpha+\delta\ne0,\qquad\alpha\gamma+\beta\delta=0;\\[2pt]
G_6:\quad &[e_1,e_2]=\alpha e_2+\beta e_3,\quad
[e_1,e_3]=\gamma e_2+\delta e_3,\quad[e_2,e_3]=0,\\
&\alpha+\delta\ne0,\qquad\alpha\gamma-\beta\delta=0;\\[2pt]
G_7:\quad &[e_1,e_2]=-\alpha e_1-\beta e_2-\beta e_3,\quad
[e_1,e_3]=\alpha e_1+\beta e_2+\beta e_3,\\
&[e_2,e_3]=\gamma e_1+\delta e_2+\delta e_3,\qquad
\alpha+\delta\ne0,\quad\alpha\gamma=0.
\end{align*}
\subsection{Affine vector-field spaces and normalization}
For fixed structure parameters, the equation is linear in the vector coordinates. If a particular solution exists, the full solution space is $V_0+\ker\mathcal A$, with $\mathcal A$ the coefficient matrix of the five steady equations. This follows by subtracting two solutions and, conversely, adding a homogeneous solution to $V_0$. Free vector directions therefore occur at rank drops, which must not be removed by premature division.

In homogeneous bracket families, simultaneous scaling of structure constants and vector coordinates by $\lambda\ne0$ scales the connection by $\lambda$ and both terms in the steady equation by $\lambda^2$. The normalizations of $G_2$ and $G_7$ are justified in this way, and the scale is restored afterwards. The $G_4$ bracket contains fixed coefficients and uses instead the discrete relation $\eta^2=1$.

\subsection{Algebraic candidates and exact lifting}
A polynomial with an isolating rational interval specifies a real algebraic number exactly. Decimal approximations serve only to identify branches. If elimination gives a factor $P(b)$ and a linear subresultant $b(D(b)a-N(b))$, it is necessary to handle $b=0$, prove that $D$ does not vanish at roots of $P$, and substitute the lifted candidates into the original equations. The condition $\gcd(D,P)=1$ settles the denominator issue. A representative $A(b)$ satisfying $DA-N\equiv0\pmod P$ then gives the unique lift, provided the original residuals vanish modulo $P$.

Thus a resultant alone is insufficient: it gives projected necessary candidates, whereas exact lifting gives sufficiency. For the elimination framework, see \cite{CoxLittleOShea}. In the $G_7$ problem, equality of reduced Gröbner bases over $\mathbb Q$, using the same lexicographic order, certifies equality of the branch ideals. Their real solutions are then determined separately.

\section{Universal adapted-frame formulas}\label{sec:universal}
Choose a pseudo-orthonormal frame with $D=\spn\{f_1,f_3\}$ and $D^\perp=\spn\{f_2\}$. Write
\begin{equation}\label{eq:bracket9}
\begin{aligned}
[f_1,f_2]&=pf_1+qf_2+rf_3,\\
[f_1,f_3]&=sf_1+tf_2+uf_3,\\
[f_2,f_3]&=vf_1+wf_2+kf_3.
\end{aligned}
\end{equation}
The symbol $u$ is unrelated to $\mu$. The coefficients are those of an actual Lie bracket and satisfy its Jacobi identity.

\subsection{Connection and curvature matrices}
\begin{lemma}\label{lem:universal}
Let the $j$th column of $A_i$ represent $\nabla^B_{f_i}f_j$, and the $l$th column of $R_{ij}$ represent $R^B(f_i,f_j)f_l$. Then
\begin{equation}\label{eq:A}
A_1=\begin{pmatrix}0&0&s\\0&q&0\\s&0&0\end{pmatrix},\quad
A_2=\begin{pmatrix}-p&0&v\\0&0&0\\-r&0&k\end{pmatrix},\quad
A_3=\begin{pmatrix}0&0&-u\\0&-w&0\\-u&0&0\end{pmatrix}.
\end{equation}
The corresponding curvature matrices are
\begin{align}
R_{12}&=\begin{pmatrix}
-sr-sv+pq&0&sk+ur-qv\\
0&-pq+rw&0\\
-sk-2sp+ur+qr&0&sr+sv-kq
\end{pmatrix},\\
R_{13}&=\begin{pmatrix}
tp&0&-s^2-tv+u^2\\
0&-sq+uw&0\\
-s^2+tr+u^2&0&-tk
\end{pmatrix},\\
R_{23}&=\begin{pmatrix}
-ur-uv+pw&0&-sv+2uk+up-vw\\
0&kw-qv&0\\
-sv-up+rw&0&ur+uv-kw
\end{pmatrix}.
\end{align}

The remaining curvature matrices satisfy $R_{ji}=-R_{ij}$ and $R_{ii}=0$.
\end{lemma}
\begin{proof}
With $[f_i,f_j]=\sum_lC_{ij}^lf_l$, $\nabla^L_{f_i}f_j=\sum_{l=1}^{3}\Gamma^l_{ij}f_l$, the Koszul formula is
\[
2\varepsilon_l\Gamma^l_{ij}=\varepsilon_lC_{ij}^l-\varepsilon_iC_{jl}^i+\varepsilon_jC_{li}^j.
\]
It applies to the two equal-distribution blocks in \eqref{eq:bott}; the mixed blocks use projected brackets. The complete derivative table is
\[
\begin{array}{lll}
\nabla^B_{f_1}f_1=sf_3,&\nabla^B_{f_1}f_2=qf_2,&\nabla^B_{f_1}f_3=sf_1,\\
\nabla^B_{f_2}f_1=-pf_1-rf_3,&\nabla^B_{f_2}f_2=0,&\nabla^B_{f_2}f_3=vf_1+kf_3,\\
\nabla^B_{f_3}f_1=-uf_3,&\nabla^B_{f_3}f_2=-wf_2,&\nabla^B_{f_3}f_3=-uf_1.
\end{array}
\]
For example, $2g(\nabla^L_{f_1}f_1,f_3)=-2s$. Since $f_3$ is timelike, its coordinate coefficient is $s$. Also $\pi_D[f_2,f_1]=-pf_1-rf_3$, giving the first column of $A_2$. The projection of $\nabla^L_{f_2}f_2$ onto the unit line $D^\perp$ is zero.

Curvature is then computed by
\begin{equation}\label{eq:matrixR}
R_{ij}=A_iA_j-A_jA_i-\sum_lC_{ij}^lA_l.
\end{equation}
The last term cannot be omitted for a noncoordinate frame. As a full example, $A_1,A_3$ commute because their restrictions to $\spn\{f_1,f_3\}$ are multiples of the same matrix, and the remaining blocks are scalar. Thus $R_{13}=-sA_1-tA_2-uA_3$. Its potentially nonzero entries are $tp$, $-s^2-tv+u^2$, $-sq+uw$, $-s^2+tr+u^2$, $-tk$ in positions $(1,1),(1,3),(2,2),(3,1),(3,3)$.

For an additional check, the $(3,1)$ entry of $A_1A_2-A_2A_1$ is $-sp-sk$. The bracket correction gives $-ps+qr+ru$, and their sum is the stated $(3,1)$ entry of $R_{12}$. Applying the same matrix formula to the other entries gives the two remaining matrices.
\end{proof}

\begin{remark}[Torsion and metric compatibility]
The torsion is $T^B(f_1,f_2)=T^B(f_2,f_3)=0$, $T^B(f_1,f_3)=-tf_2$. Thus integrability of $D$ is equivalent to $t=0$, also the torsion-free condition here. Metric compatibility is distinct: with $E=\diag(1,1,-1)$, the equations $A_i^TE+EA_i=0$ hold precisely when $p=q=k=w=0$ and $r+v=0$. These restrictions need not hold in the models under study.
\end{remark}

\subsection{Ricci tensor and the vector-field term}
\begin{lemma}\label{lem:ricci}
Writing $S_{ij}=\Ric(f_i,f_j)$, one has
\begin{align}
S_{11}&=s^2-tr-u^2,& S_{22}&=0,& S_{33}&=-s^2-tv+u^2,\\
S_{12}&=-\tfrac12sr+\tfrac12up+pq-rw,\\
S_{13}&=\tfrac12t(k+p),\\
S_{23}&=\tfrac12sk-\tfrac12uv+kw-qv.
\end{align}
For a constant-coefficient field $V=M_1f_1+M_2f_2+M_3f_3$,
\begin{align}
(\LB_Vg)_{11}&=2sM_3,&(\LB_Vg)_{22}&=0,&(\LB_Vg)_{33}&=2uM_1,\\
(\LB_Vg)_{12}&=-pM_1+qM_2+vM_3,\\
(\LB_Vg)_{13}&=-sM_1-uM_3,\\
(\LB_Vg)_{23}&=rM_1-wM_2-kM_3.
\end{align}
\end{lemma}
\begin{proof}
The identity $g(R_{il}f_j,f_l)=\varepsilon_l(R_{il})^l{}_j$ gives $\rho_{ij}=-\sum_l(R_{il})^l{}_j$. The diagonal entries immediately give those of $S$. The unsymmetrized off-diagonal entries are
\begin{align*}
\rho_{12}&=pq-rw,&\rho_{21}&=-sr+up+pq-rw,\\
\rho_{13}&=tk,&\rho_{31}&=tp,\\
\rho_{23}&=sk-uv+kw-qv,&\rho_{32}&=kw-qv.
\end{align*}
For instance $\rho_{12}=-(R_{12})^2{}_2-(R_{13})^3{}_2=pq-rw$, whereas $\rho_{21}$ has additional contributions. Averaging yields the asserted tensor. Next $\nabla^B_{f_i}V=A_i\boldsymbol M$. If $N$ has these vectors as columns, then $\LB_Vg=N^TE+EN$, which gives all six displayed entries.
\end{proof}

\begin{proposition}[Steadiness and the master system]\label{prop:steady}
Every left-invariant solution satisfies $\mu=0$. Its remaining equations are
\begin{equation}\label{eq:master}
\left\{\begin{aligned}
&sM_3+s^2-tr-u^2=0,\\
&-pM_1+qM_2+vM_3-sr+up+2pq-2rw=0,\\
&-sM_1-uM_3+t(k+p)=0,\\
&rM_1-wM_2-kM_3+sk-uv+2kw-2qv=0,\\
&uM_1-s^2-tv+u^2=0.
\end{aligned}\right.
\end{equation}
\end{proposition}
\begin{proof}
The $(2,2)$ component is $2\mu=0$. The other five independent components give the displayed equations, with the diagonal ones divided by $2$. Conversely, these equations and $\mu=0$ annihilate every entry of the symmetric soliton tensor.
\end{proof}
The constant-coefficient assumption matters: for variable $M_i$, the same component contains $2f_2(M_2)$ and gives $f_2(M_2)+\mu=0$, not steady behavior. We make no classification claim for such fields.

\subsection{The second distribution}
\begin{proposition}\label{prop:swap}
For $D_2$, take $f_1=\E_2,f_2=\E_1,f_3=-\E_3$. If $\theta$ denotes the original parameters, the new parameters $\theta'$ are
\[
\begin{array}{c|c}
G_1&(\alpha,\beta)\mapsto(-\alpha,\beta)\\
G_2,G_3,G_4,G_6&\text{all parameters unchanged, including }\eta\\
G_5,G_7&(\alpha,\beta,\gamma,\delta)\mapsto(-\alpha,\beta,\gamma,-\delta).
\end{array}
\]
The pair $(\theta;x,y,z)$ solves the $D_2$ problem if and only if $(\theta';y,x,-z)$ solves the $D_1$ problem.
\end{proposition}
\begin{proof}
The adapted frame results from applying \eqref{eq:rotation} to $(e_1,-e_2,-e_3)$. Computing brackets gives the table. For example, for $G_5$, $[e'_1,e'_3]=-\alpha e'_1+\beta e'_2$ and $[e'_2,e'_3]=\gamma e'_1-\delta e'_2$, and both restrictions are preserved. The other cases are checked in the same way.

In this frame, $V=\sqrt2yf_1+\sqrt2xf_2-zf_3$. The frame change preserves the metric, and the connection definition and soliton equation are tensorial. Applying the same operation twice restores the original data, proving both directions.
\end{proof}
With $U=\left(\begin{smallmatrix}0&1&0\\1&0&0\\0&0&-1\end{smallmatrix}\right)$, one has $U^TEU=E$ and $F^{(2)}=U^TF^{(1)}U$ after the stated substitutions. This also preserves the dimension of each affine vector-field solution space.

\section{Unimodular standard families}\label{sec:unimodular}\label{sec:classify}
All parameters and free coordinates are real, and every solution below has $\mu=0$. We first solve the $D_1$ equations; the second distribution follows from Proposition~\ref{prop:swap}.
\subsection{The family $G_1$}
\begin{theorem}
For $\alpha\ne0$, neither distribution admits a left-invariant Bott affine Ricci soliton on $G_1$.
\end{theorem}
\begin{proof}
For $D_1$, the $(1,1)$ equation after setting $\mu=0$ is
\[
\alpha^2-2\alpha\beta+2\beta^2=(\alpha-\beta)^2+\beta^2=0.
\]
Both squares must vanish over $\R$, forcing $\alpha=\beta=0$ in contradiction with the restriction. For $D_2$, the corresponding equation is $(\alpha+\beta)^2+\beta^2=0$ and gives the same contradiction. No division by a potentially vanishing parameter is used.
\end{proof}

\subsection{The family $G_2$: four real algebraic branches}
For $\sigma=\pm1$, define
\begin{align}
P_\sigma(T)&=16T^6+64\sigma T^5-96T^4+152\sigma T^3
             -144T^2+88\sigma T-23,\label{eq:P2}\\
A_\sigma(T)&=-\tfrac16\bigl(8T^5+36\sigma T^4-28T^3
                   +70\sigma T^2-51T+31\sigma\bigr).\label{eq:A2}
\end{align}

For a real root $B$ of $P_\sigma$, set
\begin{equation}\label{eq:XYZ2}
\begin{aligned}
A&=A_\sigma(B),&X&=AB+\tfrac12,\\
Z&=A+B+(A-B)X,\\
Y&=\tfrac12\bigl((A+B)Z-AB+B^2+1\bigr)-X.
\end{aligned}
\end{equation}
\begin{theorem}\label{thm:G2}
Take $\lambda\ne0$, $\sigma=\pm1$, a real root $B$ of $P_\sigma$, and
\[
(\alpha,\beta,\gamma)=\lambda(A,B,1).
\]
These are exactly the admissible structures. The unique vector coordinates are $\lambda(X,Y,Z)$ for $D_1$ and $\lambda(Y,X,-Z)$ for $D_2$. Each $P_\sigma$ has exactly two real roots, and $P_-,P_+$ are relatively prime, giving four normalized branches.
\end{theorem}
\begin{proof}
Since $\gamma\ne0$, normalize by $\lambda=\gamma$ and denote the new variables by $(a,b,X,Y,Z)$. The $(3,3)$ equation is $-2X=-2ab-1$; the $(1,3)$ equation is $Z-(a-b)X=a+b$; the $(1,2)$ equation is $X+Y-(a+b)Z/2=(-ab+b^2+1)/2$. They successively determine $X,Z,Y$ as in \eqref{eq:XYZ2}. Only constant nonzero divisors occur, so no case is lost.

The two remaining equations are equivalent to $f_2=g_2=0$, where
\begin{align}
f_2={}&a^3b-2a^2b^2+2a^2+ab^3-3ab-b^2-1,\label{eq:f2}\\
g_2={}&2a^4b-2a^3b^2+3a^3-2a^2b^3-a^2b\notag\\
&+2ab^4+9ab^2+10a-3b^3+6b.\label{eq:g2}
\end{align}
Their resultant is
\[
\operatorname{Res}_a(f_2,g_2)=-bP_-(b)P_+(b).
\]
The factor $b$ must be examined independently. For $b=0$, the equations are $2a^2-1=0$ and $3a^3+10a=0$. The former reduces the latter to $23a/2=0$, impossible. Hence $b\ne0$.

The final subresultant linear in $a$ is $b(D_2a-N_2)$, with
\begin{align*}
D_2(b)&=288b^6+416b^4+64b^2-46,\\
N_2(b)&=16b^7-96b^5-96b^3+43b.
\end{align*}
Exact Euclidean division gives $\gcd(D_2,P_-P_+)=1$. Thus each candidate root has a unique lift in $a$. For each sign,
\[
D_2A_\sigma-N_2\equiv0\pmod{P_\sigma},\qquad
f_2(A_\sigma(b),b)\equiv g_2(A_\sigma(b),b)\equiv0\pmod{P_\sigma}.
\]
Every real candidate therefore produces a solution of the original system. Direct reduction of all original tensor components gives the same zero residuals. The three linear eliminations already made the vector coordinates unique. Section~\ref{sec:roots} certifies the root counts. Finally restore $\lambda$ and apply the distribution correspondence.
\end{proof}

\subsection{The family $G_3$: free directions}
\begin{theorem}
Solutions exist exactly when $\alpha=\beta=0$. Then the full $D_1$ solution set is $\gamma x=0$, with $y,z$ arbitrary; the full $D_2$ set is $\gamma y=0$, with $x,z$ arbitrary. If also $\gamma=0$, the algebra is abelian and every left-invariant vector field is a solution.
\end{theorem}
\begin{proof}
The $(3,3)$ component gives $\alpha\beta=0$. If $\alpha\ne0,\beta=0$, the $(1,2)$ component gives $z=\gamma$ and the $(1,1)$ component then forces $\alpha^2=0$, a contradiction. If $\alpha=0,\beta\ne0$, the same steps give $z=-\gamma$ and $\beta^2=0$, again impossible. Hence $\alpha=\beta=0$. Substitution into all components leaves precisely $\gamma x=0$, proving both directions. The frame correspondence gives the $D_2$ result.
\end{proof}

For fixed $\gamma\ne0$, there are two free vector coordinates; in the abelian case there are three. For example, with $\gamma=1$, the sole adapted bracket is $[f_1,f_2]=f_3$ and the sole connection derivative is $\nabla^B_{f_2}f_1=-f_3$. All curvature matrices vanish. Thus a nonabelian Lie algebra can carry a flat Bott connection for the selected distribution.

\subsection{The family $G_4$: six branches for each sign}
Define
\begin{align}
Q_1(T)&=16T^6-96T^5+160T^4-56T^3-56T^2+24T+9,\\
Q_2(T)&=16T^6+32T^5-352T^4+760T^3-744T^2+360T-71,\\
B_1(T)&=-\tfrac16(16T^5-96T^4+172T^3-116T^2+10T+15),\\
B_2(T)&=\tfrac1{14}(48T^5+160T^4-852T^3+1116T^2-562T+81).
\end{align}

For $j=1,2$ and a real root $B$ of $Q_j$, let
\begin{equation}\label{eq:XYZ4}
\begin{aligned}
A&=B_j(B),&X&=-AB-\tfrac12,\\
Z&=A+B-(A-B)X,\\
Y&=\tfrac12\bigl(AB-2A-B^2+2B-(A+B)Z-1\bigr)-X.
\end{aligned}
\end{equation}
\begin{theorem}\label{thm:G4}
For each $\eta=\pm1$, the admissible structures are exactly $\alpha=\eta A$, $\beta=\eta B$ with the quantities above. The unique vector coordinates are $(X,Y,\eta Z)$ for $D_1$ and $(Y,X,-\eta Z)$ for $D_2$. The relatively prime polynomials $Q_1,Q_2$ have four and two real roots, respectively, hence six branches for each $\eta$.
\end{theorem}
\begin{proof}
Put $\alpha=\eta a,\beta=\eta b,z=\eta Z,x=X,y=Y$. The identity $\eta^2=1$ reduces the equations to the same $\eta=1$ system. In particular, $x,y$ do not change sign with $\eta$. The $(3,3)$ equation gives $X=-ab-1/2$. The next two linear equations determine $Z,Y$ as displayed. The two residual equations are
\[
f_4=f_2+2a+2b=0,\qquad g_4=g_2+4a^2-24ab+4b^2-8=0.
\]
With this ordering and normalization, exact elimination yields
\begin{equation}\label{eq:g4resultant}
\operatorname{Res}_a(f_4,g_4)=-bQ_1(b)Q_2(b).
\end{equation}

At $b=0$, the first equation is $h(a)=2a^2+2a-1=0$, and the remainder of $g_4(a,0)$ modulo $h$ is $(21a-15)/2$. A common zero would be $a=5/7$, whereas $h(5/7)=71/49\ne0$. Hence $b=0$ is not a solution.

The last linear subresultant is $b(Da-N)$, where
\begin{align*}
D(b)&=288b^6-1088b^5+1440b^4-704b^3-64b^2+180b-42,\\
N(b)&=16b^7-32b^6-96b^5+320b^4-288b^3+52b^2+55b-30.
\end{align*}
One obtains $\gcd(D,Q_1Q_2)=1$ and, for each $j$,
\[
DB_j-N\equiv0\pmod{Q_j},\qquad
f_4(B_j(b),b)\equiv g_4(B_j(b),b)\equiv0\pmod{Q_j}.
\]
This proves that the lifting is unique and that every real candidate yields a solution of the original system. The root counts follow in Section~\ref{sec:roots}. Restoring the sign parameter and applying Proposition~\ref{prop:swap} completes the classification.
\end{proof}

\section{Nonunimodular standard families}\label{sec:nonunimodular}
\subsection{The family $G_5$}
\begin{theorem}\label{thm:G5}
For any $\lambda\ne0$, the full solution sets are
\begin{center}\begin{tabular}{ccc}\toprule
Distribution&$(\alpha,\beta,\gamma,\delta)$&$(x,y,z)$\\\midrule
$D_1$&$(\lambda,-\lambda,0,0)$&$(x,0,z)$\\
$D_1$&$(0,0,-\lambda,\lambda)$&$(x,0,z)$\\
$D_2$&$(\lambda,\lambda,0,0)$&$(0,y,z)$\\
$D_2$&$(0,0,\lambda,\lambda)$&$(0,y,z)$\\\bottomrule
\end{tabular}\end{center}
All coordinates not fixed in the table are arbitrary real numbers.
\end{theorem}
\begin{proof}
The adapted coefficients are
\begin{align*}
s&=(\alpha+\beta+\gamma+\delta)/2,&t&=(\alpha-\beta+\gamma-\delta)/2,\\
v&=(\alpha+\beta-\gamma-\delta)/2,&w&=(\alpha-\beta-\gamma+\delta)/2,
\end{align*}
with $p=q=r=u=k=0$. The inverse transformation is
\begin{align*}
\alpha&=(s+t+v+w)/2,&\beta&=(s-t+v-w)/2,\\
\gamma&=(s+t-v-w)/2,&\delta&=(s-t-v+w)/2.
\end{align*}
Thus $\alpha+\delta=s+w$ and $\alpha\gamma+\beta\delta=(s^2+t^2-v^2-w^2)/2$. The restrictions become $s+w\ne0$ and $s^2+t^2-v^2-w^2=0$. The master system is
\[
s(s+z)=0,\quad vz=0,\quad sx=0,\quad wy=0,\quad s^2+tv=0.
\]
If $s\ne0$, then $z=-s\ne0$, hence $v=0$, and the last equation forces $s^2=0$, impossible. Therefore $s=0,w\ne0$. From $tv=0$ and $t^2=v^2+w^2$ we obtain $t\ne0$, then $v=0$, then $t=\pm w$. Inverting the parameter change gives $(w,-w,0,0)$ when $t=w$, and $(0,0,-w,w)$ when $t=-w$. The only vector constraint is $wy=0$, so $y=0$ while $x,z$ are free. All remaining equations vanish. The second distribution follows by the proved correspondence.
\end{proof}

The symmetrized Ricci tensor vanishes in these families; in particular $V=0$ is included. For the example $(1,-1,0,0)$, the only nonzero connection derivative is $\nabla^B_{f_3}f_2=-f_2$, and the full curvature vanishes. Since \([f_1,f_3]=\lambda f_2\ne 0\), these observations do not require the group to be abelian. 

\subsection{The family $G_6$}
Let $r_0=(1-\sqrt5)/2<0$.
\begin{theorem}\label{thm:G6}
The full $D_1$ solution set consists of the following disjoint cases:
\begin{enumerate}
\item $\delta=\gamma=\beta=0$, $\alpha\ne0$, and $(x,y,z)=(q,\alpha-q,z)$ with $q,z\in\R$;
\item $\delta=\gamma=0$, $\alpha\ne0$, $\beta\ne0$, and $(x,y,z)=(0,\alpha,z)$ with $z\in\R$;
\item $\delta=\lambda\ne0$, $\alpha=\beta=0$, $\gamma=\epsilon\sqrt2\lambda$, $\epsilon=\pm1$, and $(x,y,z)=(-\lambda/2,\lambda,0)$;
\item $\delta=\lambda\ne0$, $\alpha=r_0\lambda$, $\gamma=\epsilon\lambda\sqrt{-2r_0}$, $\beta=r_0\gamma$, $\epsilon=\pm1$, and
\[
(x,y,z)=\left(-\lambda/2,\lambda(1+r_0/2),\beta/2\right).
\]
\end{enumerate}
For $D_2$, all structure conditions are unchanged and each triple is replaced by $(y,x,-z)$.
\end{theorem}
\begin{proof}
If $\delta=0$, then $\alpha\ne0$ and $\alpha\gamma=\beta\delta$ implies $\gamma=0$. The remaining system is exactly $x+y=\alpha$, $\beta x=0$. Splitting $\beta=0$ and $\beta\ne0$ gives cases 1 and 2, including all free directions.

Suppose $\delta\ne0$. The $(3,3)$ equation gives $x=-\delta/2$. Multiplication of the $(1,3)$ equation by $4$ gives
\[
-2\gamma x-2\delta z=-\gamma(\alpha-\delta).
\]
Substitution cancels the $\gamma\delta$ terms, yielding $z=\alpha\gamma/(2\delta)=\beta/2$. The $(1,1)$ equation becomes $\gamma(\gamma-\beta)=2\delta^2$, so $\gamma\ne0$ and the next division is valid. The $(2,3)$ component gives $y=\delta+\alpha/2$. Substitution into the $(1,2)$ component gives $2\alpha^2+\beta\gamma=0$; multiplying by $\delta$ and using the structure constraint yields
\[
\alpha(2\alpha\delta+\gamma^2)=0.
\]
For $\alpha=0$, we have $\beta=0$ and $\gamma^2=2\delta^2$, which is case 3. For $\alpha\ne0$, we have $\gamma^2=-2\alpha\delta>0$. The equation $\gamma(\gamma-\beta)=2\delta^2$ then yields $\alpha^2-\alpha\delta-\delta^2=0$. Thus $r=\alpha/\delta$ satisfies $r^2-r-1=0$ and $r<0$. Only $r_0$ is admissible; the positive root would force a negative value for $\gamma^2$. This gives case 4.

Conversely, direct substitution verifies every listed case and the structure restrictions. In case 4, $r_0+1\ne0$ ensures $\alpha+\delta\ne0$. Both signs of $\gamma$ are accounted for by $\epsilon$. The transformation gives the $D_2$ classification.
\end{proof}
The $\delta=0$ families are solved from the original equations, rather than by substituting into formulas with $1/\delta$. Their larger vector-field spaces reflect genuine rank drops.

\subsection{The family $G_7$}
Define
\begin{equation}\label{eq:quartic}
H(C)=C^4-3C^2+2C-1.
\end{equation}
Let $\mathcal T$ contain the tuples $(B,C,X,Y,Z)$ specified by:
\begin{enumerate}
\item[(a)] $B=-1$, $C=0$, $X,Y\in\R$, $Z=X+Y+1$;
\item[(b)] $B=-1\pm\sqrt2$, $C=0$, and
\[
X=\frac{B-1}{2(B+1)},\quad Y=-1,\quad Z=\frac{B}{B+1};
\]
\item[(c)] $B=0$, $C=-1\pm\sqrt2$, $(X,Y,Z)=(C/2,-1,0)$;
\item[(d)] $C$ is either real root of $H$, and
\begin{equation}\label{eq:G7quartic}
\begin{aligned}
B&=-C^3+3C-2,&X&=(1-C^2)/2,\\
Y&=-(C^3+C^2-2C+3)/2,&Z&=(1-C-C^2)/2.
\end{aligned}
\end{equation}
\end{enumerate}
\begin{theorem}\label{thm:G7}
For every $\lambda\ne0$, all solutions are as follows.
\begin{enumerate}
\item If $\alpha\ne0$, then necessarily $\gamma=\delta=0$. For $D_1$,
\[
\alpha=\lambda,\quad\beta=\lambda(1\pm\sqrt2),\quad(x,y,z)=\lambda(-1/2,-1,-1).
\]
For $D_2$,
\[
\alpha=\lambda,\quad\beta=-\lambda(1\pm\sqrt2),\quad(x,y,z)=\lambda(1,1/2,-1).
\]
\item If $\alpha=0$, then $\delta=\lambda\ne0$. For every tuple in $\mathcal T$, the $D_1$ solution is
\[
\beta=\lambda B,\quad\gamma=\lambda C,\quad(x,y,z)=\lambda(X,Y,Z).
\]
The $D_2$ solution is
\[
\beta=-\lambda B,\quad\gamma=-\lambda C,\quad(x,y,z)=\lambda(-Y,-X,Z).
\]
\end{enumerate}
\end{theorem}
\begin{proof}
\emph{Step 1: $\alpha\ne0$.} The constraint $\alpha\gamma=0$ gives $\gamma=0$. Normalize all structure and vector coordinates by $\alpha$, defining $b=\frac{\beta}{\alpha}$, $d=\frac{\delta}{\alpha}$, and $(X,Y,Z)=\frac{1}{\alpha}(x,y,z)$.  Over $\mathbb Q$, the five equations generate the same ideal as
\begin{equation}\label{eq:G7gb1}
\begin{cases}
2X+1-bd-d^2+d=0,\\
2Y+2+3d=0,\\
2Z+2-bd-d^2+d=0,\\
b^2+2bd-2b-d^2-2d-1=0,\\
d(d^2+d+1)=0.
\end{cases}
\end{equation}
The equality is certified by reducing both generating sets to their reduced lexicographic Gröbner bases with order $X>Y>Z>b>d$. Thus the replacement preserves the original zero set in both directions, rather than retaining only elimination consequences.

Since $d^2+d+1=(d+1/2)^2+3/4>0$ over $\R$, the final equation implies $d=0$. Next $b^2-2b-1=0$ gives $b=1\pm\sqrt2$, and the first three equations give $(X,Y,Z)=(-1/2,-1,-1)$. 

\emph{Step 2: $\alpha=0$.} The trace restriction gives $\delta\ne0$. Normalize to $(\alpha,\beta,\gamma,\delta)=(0,B,C,1)$. Lexicographic elimination in $X>Y>Z>B>C$ gives
\begin{equation}\label{eq:G7elim}
C^2(C^2+2C-1)H(C)=0.
\end{equation}
Every solution therefore falls into $C=0$, $J(C)=C^2+2C-1=0$, or $H(C)=0$. We add each relation to the original ideal instead of dividing by any factor.

For $C=0$, an equivalent generating set is
\begin{equation}\label{eq:G7c0}
\begin{cases}
X+Y-Z+B^2/4+B/2+5/4=0,\\
(B+1)(Y+1)=0,\\
B^2/2+(B+1)Z-1/2=0,\\
(B+1)(B^2+2B-1)=0.
\end{cases}
\end{equation}
If $B=-1$, the only remaining condition is $Z=X+Y+1$, leaving two free coordinates. Otherwise $B^2+2B-1=0$, so $B+1\ne0$ and $Y=-1$. The third equation gives $Z=(1-B^2)/(2(B+1))=B/(B+1)$; the first then gives $X=(B-1)/(2(B+1))$. These are (a) and (b). Dividing first by $B+1$ would lose the entire family (a).

For $J(C)=0$, the branch ideal has the equivalent triangular basis
\[
X-C/2=0,\quad Y+1=0,\quad Z=0,\quad B=0,\quad J(C)=0,
\]
which gives (c). For $H(C)=0$, an equivalent basis is
\[
\begin{gathered}
X+(C^2-1)/2=0,\quad Y+(C^3+C^2-2C+3)/2=0,\\
Z+(C^2+C-1)/2=0,\quad B+C^3-3C+2=0,\quad H(C)=0.
\end{gathered}
\]
This gives (d). Each equality of ideals is checked in the same polynomial ring over $\mathbb Q$, ensuring sufficiency as well as necessity. 

The three conditions on $C$ are disjoint. Neither $J$ nor $H$ vanishes at $0$. Moreover $H\equiv1-4C\pmod J$, so a common root would be $C=1/4$, whereas $J(1/4)=-7/16\ne0$. On the quartic branch, $BC=-1$ follows from $H=0$; hence $B,C$ are nonzero. Section~\ref{sec:roots} shows that there are exactly two real quartic roots.

Restoring the scale yields all $D_1$ solutions. The parameter transformation for $D_2$ reverses both $\alpha$ and $\delta$. Keeping track of this sign in the normalization gives the stated changes in $\beta,\gamma$ and vector coordinates.
\end{proof}

\section{Certified real roots and exhaustiveness}\label{sec:roots}
\subsection{Sturm sequences}
For a squarefree polynomial $F\in\mathbb Q[T]$, form $S_0=F$, $S_1=F'$ and $S_{i+1}=-\operatorname{rem}(S_{i-1},S_i)$ until a nonzero constant is reached. Let $\mathcal V(t)$ be the sign-variation count after zero entries are removed. Sturm's theorem gives $\mathcal V(a)-\mathcal V(b)$ roots in $(a,b)$ when the endpoints are not roots. Positive rescaling of chain members preserves the sign counts.

The exact counts are
\begin{center}\begin{tabular}{cccc}\toprule
Polynomial&$\mathcal V(-\infty)$&$\mathcal V(+\infty)$&Real roots\\\midrule
$P_-$&4&2&2\\$P_+$&4&2&2\\$Q_1$&5&1&4\\$Q_2$&4&2&2\\$H$&3&1&2\\\bottomrule
\end{tabular}\end{center}
Every finite decimal below is an exact rational endpoint. Each interval contains one root, and the intervals exhaust the real roots:
\begin{center}\small\begin{tabular}{cp{112mm}}\toprule
Polynomial&Isolating intervals\\\midrule
$P_-$&$(-0.493,-0.492)$, $(5.474,5.475)$\\
$P_+$&$(-5.475,-5.474)$, $(0.492,0.493)$\\
$Q_1$&$(-0.439,-0.438)$, $(-0.292,-0.291)$, $(1.318,1.319)$, $(3.497,3.498)$\\
$Q_2$&$(-6.599,-6.598)$, $(0.719,0.720)$\\
$H$&$(-2.053,-2.052)$, $(1.447,1.448)$\\\bottomrule
\end{tabular}\end{center}
As an explicit example, a Sturm chain for $H$ is
\[
H,\quad4C^3-6C+2,\quad\tfrac32C^2-\tfrac32C+1,
\quad\tfrac{14}{3}C+\tfrac23,\quad-\tfrac{61}{49}.
\]
The signs at $-\infty$ are $(+,-,+,-,-)$, with three variations; at $+\infty$ they are $(+,+,+,+,-)$, with one variation. There are therefore two real roots. They are approximately $-2.052300344614$ and $1.447622986855$, but their exact definitions remain the polynomial and intervals.

For the sextics, rational polynomial division gives chains ending in nonzero constants, proving squarefreeness. Evaluating the chains at each rational endpoint certifies the isolated roots individually. Their number equals the count on the whole real line, excluding any additional real roots. The exact endpoint variation counts and sign patterns are recorded in the computational appendix.

\subsection{Logical completeness of the classification}
Every solution first satisfies the universal tensor system. The elementary cases are exhausted by zero/nonzero splits, retaining all structure restrictions. In $G_2,G_4$, the resultant captures all projected candidates, exceptional factors are excluded separately, denominator coprimality gives unique lifts, and exact residual reduction proves sufficiency. In $G_7$, the elimination polynomial covers all possible real values of $C$, and each branch is described by an equal triangular ideal. The remaining one-variable equations are certified by Sturm counting.

The loci $\alpha=0$, $\beta=0$, $\gamma=0$, $\delta=0$ and $B+1=0$ have been retained whenever relevant. Generic invertibility of a coefficient matrix would not justify removing these loci. In particular, the affine vector-field families in $G_3,G_5,G_6,G_7$ are genuine parts of the solution set, not merely optional choices of representative.

\section{Geometric interpretation and concluding remarks}
\subsection{Dimensions at fixed structures}
For each admissible fixed structure, the vector field is unique in $G_2,G_4$, the $\delta\ne0$ cases of $G_6$, and the $G_7$ branches other than (a). In $G_3$, there are two free coordinates when $\gamma\ne0$, and three in the abelian case. Each $G_5$ family has two free coordinates. In $G_6$ with $\delta=\gamma=0$, the dimensions are two if $\beta=0$ and one if $\beta\ne0$. The $G_7$ family (a) has dimension two. These dimensions do not count scale or algebraic-root choices. They are preserved by the correspondence between distributions.

\subsection{A nontrivial balance of tensors}
Choose $G_6$ with $\alpha=\beta=0$, $\delta=1$, $\gamma=\sqrt2$ and $(x,y,z)=(-1/2,1,0)$. The adapted coefficients are
\[
s=v=u=k=1/\sqrt2,\quad t=w=-1/\sqrt2,\quad p=q=r=0.
\]
The nonzero independent entries of $S$ are $S_{33}=1/2$, $S_{13}=-1/4$, $S_{23}=-1/2$. The corresponding entries of $\LB_Vg$ are $-1$, $1/2$ and $1$. All other independent entries vanish. Thus $\LB_Vg+2S=0$ although both tensors are nonzero. Unlike the flat examples in $G_3$ and $G_5$, this illustrates a nonparallel vector field balancing nonzero symmetrized Ricci curvature.

\subsection{Scope and further problems}
Steadiness here concerns the Bott affine equation and does not identify these metrics with steady solutions of the ordinary Ricci-flow soliton equation. Nor have parameter points been identified under metric-preserving automorphisms or isometries. That would require an additional analysis of the simultaneous action on the metric and distribution.

Natural extensions include a general angle in the spacelike frame, nonconstant vector-field coefficients, and geometric equivalence among the listed parameter sets. A general angle introduces further exceptional rank loci. Non-left-invariant fields satisfy differential equations and do not obey the simple diagonal steady argument. These are separate problems, not extensions asserted by the present results.

\appendix
\section{Connection matrices and component systems}\label{app:systems}
All matrices in this appendix are written in the original rotated frame $\{\E_i\}$. The $j$th column of $A_i$ represents $\nabla^B_{\E_i}\E_j$. Vector coordinates are those of \eqref{eq:xyz}. Put
\[
F_{ij}=(\LB_Vg+2\Ric+2\mu g)_{ij}.
\]
For shorter expressions, each system lists the equations corresponding, in order, to
\[
F_{11},\quad F_{12},\quad F_{13}/\sqrt2,\quad F_{22},\quad F_{23}/\sqrt2,\quad F_{33}.
\]
Only nonzero constant row scalings have been used. The quantities $K_j$ on the right depend only on the structure constants and are defined locally in each subsection. Full curvature tensors can also be recovered from the displayed connection matrices, the transformed brackets and \eqref{eq:matrixR}.
\clearpage\subsection{$G_1$; distribution $D_1$}
\begingroup\small The connection matrices are\begin{gather*}
A_1=\begin{pmatrix}0 & 0 & 0\\0 & - \frac{\sqrt{2} \alpha}{2} & 0\\0 & 0 & 0\end{pmatrix}\\[3pt]A_2=\begin{pmatrix}\frac{\sqrt{2} \alpha}{2} & 0 & - \alpha - \beta\\0 & 0 & 0\\- \beta & 0 & - \frac{\sqrt{2} \alpha}{2}\end{pmatrix}\\[3pt]A_3=\begin{pmatrix}0 & 0 & - \frac{\sqrt{2} \alpha}{2}\\0 & 0 & 0\\- \frac{\sqrt{2} \alpha}{2} & 0 & 0\end{pmatrix}
\end{gather*}The six component equations are\[\left\{\begin{aligned}
2 \mu&=K_{1},\\[3pt]\alpha x - \alpha y + z \left(- \alpha - \beta\right)&=K_{2},\\[3pt]- \frac{\alpha z}{2}&=K_{3},\\[3pt]2 \mu&=0,\\[3pt]\frac{\alpha z}{2} + \beta x&=K_{5},\\[3pt]2 \alpha x - 2 \mu&=K_{6}.\end{aligned}\right.\]
Here the symbols $K_j$ are local to this subsection:\begin{align*}
K_{1}&=\alpha^{2} - 2 \alpha \beta + 2 \beta^{2},\\[3pt]K_{2}&=- \frac{\alpha^{2}}{2},\\[3pt]K_{3}&=- \alpha \left(\alpha - \beta\right),\\[3pt]K_{5}&=\frac{\alpha \left(\alpha + \beta\right)}{2},\\[3pt]K_{6}&=\alpha^{2} - 2 \beta^{2}.\end{align*}
The structure restrictions are $\alpha\ne0$.\endgroup
\clearpage\subsection{$G_1$; distribution $D_2$}
\begingroup\small The connection matrices are\begin{gather*}
A_1=\begin{pmatrix}0 & 0 & 0\\0 & - \frac{\sqrt{2} \alpha}{2} & - \alpha + \beta\\0 & \beta & \frac{\sqrt{2} \alpha}{2}\end{pmatrix}\\[3pt]A_2=\begin{pmatrix}\frac{\sqrt{2} \alpha}{2} & 0 & 0\\0 & 0 & 0\\0 & 0 & 0\end{pmatrix}\\[3pt]A_3=\begin{pmatrix}0 & 0 & 0\\0 & 0 & \frac{\sqrt{2} \alpha}{2}\\0 & \frac{\sqrt{2} \alpha}{2} & 0\end{pmatrix}
\end{gather*}The six component equations are\[\left\{\begin{aligned}
2 \mu&=0,\\[3pt]\alpha x - \alpha y + z \left(- \alpha + \beta\right)&=K_{2},\\[3pt]- \frac{\alpha z}{2} - \beta y&=K_{3},\\[3pt]2 \mu&=K_{4},\\[3pt]\frac{\alpha z}{2}&=K_{5},\\[3pt]- 2 \alpha y - 2 \mu&=K_{6}.\end{aligned}\right.\]
Here the symbols $K_j$ are local to this subsection:\begin{align*}
K_{2}&=- \frac{\alpha^{2}}{2},\\[3pt]K_{3}&=- \frac{\alpha \left(\alpha - \beta\right)}{2},\\[3pt]K_{4}&=\alpha^{2} + 2 \alpha \beta + 2 \beta^{2},\\[3pt]K_{5}&=\alpha \left(\alpha + \beta\right),\\[3pt]K_{6}&=\alpha^{2} - 2 \beta^{2}.\end{align*}
The structure restrictions are $\alpha\ne0$.\endgroup
\clearpage\subsection{$G_2$; distribution $D_1$}
\begingroup\small The connection matrices are\begin{gather*}
A_1=\begin{pmatrix}0 & 0 & \frac{\alpha}{2} - \frac{\beta}{2}\\0 & \frac{\sqrt{2} \gamma}{2} & 0\\\frac{\alpha}{2} - \frac{\beta}{2} & 0 & 0\end{pmatrix}\\[3pt]A_2=\begin{pmatrix}\frac{\sqrt{2} \gamma}{2} & 0 & - \frac{\alpha}{2} - \frac{\beta}{2}\\0 & 0 & 0\\- \beta & 0 & - \frac{\sqrt{2} \gamma}{2}\end{pmatrix}\\[3pt]A_3=\begin{pmatrix}0 & 0 & \frac{\sqrt{2} \gamma}{2}\\0 & \frac{\alpha}{2} - \frac{\beta}{2} & 0\\\frac{\sqrt{2} \gamma}{2} & 0 & 0\end{pmatrix}
\end{gather*}The six component equations are\[\left\{\begin{aligned}
2 \mu + z \left(\alpha - \beta\right)&=K_{1},\\[3pt]\gamma x + \gamma y - \frac{z \left(\alpha + \beta\right)}{2}&=K_{2},\\[3pt]\frac{\gamma z}{2} - \frac{x \left(\alpha - \beta\right)}{2}&=K_{3},\\[3pt]2 \mu&=0,\\[3pt]\beta x + \frac{\gamma z}{2} + \frac{y \left(\alpha - \beta\right)}{2}&=K_{5},\\[3pt]- 2 \gamma x - 2 \mu&=K_{6}.\end{aligned}\right.\]
Here the symbols $K_j$ are local to this subsection:\begin{align*}
K_{1}&=- \frac{\alpha^{2} - 4 \alpha \beta - \beta^{2} - 2 \gamma^{2}}{2},\\[3pt]K_{2}&=- \frac{\alpha \beta - \beta^{2} - \gamma^{2}}{2},\\[3pt]K_{3}&=\frac{\gamma \left(\alpha + \beta\right)}{2},\\[3pt]K_{5}&=- \frac{\alpha \gamma}{2},\\[3pt]K_{6}&=- 2 \alpha \beta - \gamma^{2}.\end{align*}
The structure restrictions are $\gamma\ne0$.\endgroup
\clearpage\subsection{$G_2$; distribution $D_2$}
\begingroup\small The connection matrices are\begin{gather*}
A_1=\begin{pmatrix}0 & 0 & 0\\0 & \frac{\sqrt{2} \gamma}{2} & \frac{\alpha}{2} + \frac{\beta}{2}\\0 & \beta & - \frac{\sqrt{2} \gamma}{2}\end{pmatrix}\\[3pt]A_2=\begin{pmatrix}\frac{\sqrt{2} \gamma}{2} & 0 & 0\\0 & 0 & - \frac{\alpha}{2} + \frac{\beta}{2}\\0 & - \frac{\alpha}{2} + \frac{\beta}{2} & 0\end{pmatrix}\\[3pt]A_3=\begin{pmatrix}- \frac{\alpha}{2} + \frac{\beta}{2} & 0 & 0\\0 & 0 & \frac{\sqrt{2} \gamma}{2}\\0 & \frac{\sqrt{2} \gamma}{2} & 0\end{pmatrix}
\end{gather*}The six component equations are\[\left\{\begin{aligned}
2 \mu&=0,\\[3pt]\gamma x + \gamma y + \frac{z \left(\alpha + \beta\right)}{2}&=K_{2},\\[3pt]- \beta y + \frac{\gamma z}{2} - \frac{x \left(\alpha - \beta\right)}{2}&=K_{3},\\[3pt]2 \mu + z \left(- \alpha + \beta\right)&=K_{4},\\[3pt]\frac{\gamma z}{2} + \frac{y \left(\alpha - \beta\right)}{2}&=K_{5},\\[3pt]- 2 \gamma y - 2 \mu&=K_{6}.\end{aligned}\right.\]
Here the symbols $K_j$ are local to this subsection:\begin{align*}
K_{2}&=- \frac{\alpha \beta - \beta^{2} - \gamma^{2}}{2},\\[3pt]K_{3}&=\frac{\alpha \gamma}{2},\\[3pt]K_{4}&=- \frac{\alpha^{2} - 4 \alpha \beta - \beta^{2} - 2 \gamma^{2}}{2},\\[3pt]K_{5}&=- \frac{\gamma \left(\alpha + \beta\right)}{2},\\[3pt]K_{6}&=- 2 \alpha \beta - \gamma^{2}.\end{align*}
The structure restrictions are $\gamma\ne0$.\endgroup
\clearpage\subsection{$G_3$; distribution $D_1$}
\begingroup\small The connection matrices are\begin{gather*}
A_1=\begin{pmatrix}0 & 0 & \frac{\alpha}{2} - \frac{\beta}{2}\\0 & 0 & 0\\\frac{\alpha}{2} - \frac{\beta}{2} & 0 & 0\end{pmatrix}\\[3pt]A_2=\begin{pmatrix}0 & 0 & - \frac{\alpha}{2} - \frac{\beta}{2}\\0 & 0 & 0\\- \gamma & 0 & 0\end{pmatrix}\\[3pt]A_3=\begin{pmatrix}0 & 0 & 0\\0 & \frac{\alpha}{2} - \frac{\beta}{2} & 0\\0 & 0 & 0\end{pmatrix}
\end{gather*}The six component equations are\[\left\{\begin{aligned}
2 \mu + z \left(\alpha - \beta\right)&=K_{1},\\[3pt]- \frac{z \left(\alpha + \beta\right)}{2}&=K_{2},\\[3pt]- \frac{x \left(\alpha - \beta\right)}{2}&=0,\\[3pt]2 \mu&=0,\\[3pt]\gamma x + \frac{y \left(\alpha - \beta\right)}{2}&=0,\\[3pt]- 2 \mu&=K_{6}.\end{aligned}\right.\]
Here the symbols $K_j$ are local to this subsection:\begin{align*}
K_{1}&=\begin{aligned} &- \frac{\alpha^{2}}{2} + \alpha \beta + \alpha \gamma - \frac{\beta^{2}}{2}\\ &\quad +\beta \gamma\end{aligned},\\[3pt]K_{2}&=- \frac{\gamma \left(\alpha - \beta\right)}{2},\\[3pt]K_{6}&=- 2 \alpha \beta.\end{align*}
The structure restrictions are $\alpha,\beta,\gamma\in\R$.\endgroup
\clearpage\subsection{$G_3$; distribution $D_2$}
\begingroup\small The connection matrices are\begin{gather*}
A_1=\begin{pmatrix}0 & 0 & 0\\0 & 0 & \frac{\alpha}{2} + \frac{\beta}{2}\\0 & \gamma & 0\end{pmatrix}\\[3pt]A_2=\begin{pmatrix}0 & 0 & 0\\0 & 0 & - \frac{\alpha}{2} + \frac{\beta}{2}\\0 & - \frac{\alpha}{2} + \frac{\beta}{2} & 0\end{pmatrix}\\[3pt]A_3=\begin{pmatrix}- \frac{\alpha}{2} + \frac{\beta}{2} & 0 & 0\\0 & 0 & 0\\0 & 0 & 0\end{pmatrix}
\end{gather*}The six component equations are\[\left\{\begin{aligned}
2 \mu&=0,\\[3pt]\frac{z \left(\alpha + \beta\right)}{2}&=K_{2},\\[3pt]- \gamma y - \frac{x \left(\alpha - \beta\right)}{2}&=0,\\[3pt]2 \mu + z \left(- \alpha + \beta\right)&=K_{4},\\[3pt]\frac{y \left(\alpha - \beta\right)}{2}&=0,\\[3pt]- 2 \mu&=K_{6}.\end{aligned}\right.\]
Here the symbols $K_j$ are local to this subsection:\begin{align*}
K_{2}&=- \frac{\gamma \left(\alpha - \beta\right)}{2},\\[3pt]K_{4}&=\begin{aligned} &- \frac{\alpha^{2}}{2} + \alpha \beta + \alpha \gamma - \frac{\beta^{2}}{2}\\ &\quad +\beta \gamma\end{aligned},\\[3pt]K_{6}&=- 2 \alpha \beta.\end{align*}
The structure restrictions are $\alpha,\beta,\gamma\in\R$.\endgroup
\clearpage\subsection{$G_4$; distribution $D_1$}
\begingroup\small The connection matrices are\begin{gather*}
A_1=\begin{pmatrix}0 & 0 & \frac{\alpha}{2} - \frac{\beta}{2}\\0 & - \frac{\sqrt{2}}{2} & 0\\\frac{\alpha}{2} - \frac{\beta}{2} & 0 & 0\end{pmatrix}\\[3pt]A_2=\begin{pmatrix}- \frac{\sqrt{2}}{2} & 0 & - \frac{\alpha}{2} - \frac{\beta}{2}\\0 & 0 & 0\\- \beta + 2 \eta & 0 & \frac{\sqrt{2}}{2}\end{pmatrix}\\[3pt]A_3=\begin{pmatrix}0 & 0 & - \frac{\sqrt{2}}{2}\\0 & \frac{\alpha}{2} - \frac{\beta}{2} & 0\\- \frac{\sqrt{2}}{2} & 0 & 0\end{pmatrix}
\end{gather*}The six component equations are\[\left\{\begin{aligned}
2 \mu + z \left(\alpha - \beta\right)&=K_{1},\\[3pt]- x - y - \frac{z \left(\alpha + \beta\right)}{2}&=K_{2},\\[3pt]- \frac{x \left(\alpha - \beta\right)}{2} - \frac{z}{2}&=K_{3},\\[3pt]2 \mu&=0,\\[3pt]x \left(\beta - 2 \eta\right) + \frac{y \left(\alpha - \beta\right)}{2} - \frac{z}{2}&=K_{5},\\[3pt]- 2 \mu + 2 x&=K_{6}.\end{aligned}\right.\]
Here the symbols $K_j$ are local to this subsection:\begin{align*}
K_{1}&=\begin{aligned} &- \frac{\alpha^{2}}{2} - 2 \alpha \eta + \frac{\beta^{2}}{2} + 1\\ &\quad +2 \alpha \beta - 2 \beta \eta\end{aligned},\\[3pt]K_{2}&=\begin{aligned} &\alpha \eta + \frac{\beta^{2}}{2} - \beta \eta + \frac{1}{2}\\ &\quad - \frac{\alpha \beta}{2}\end{aligned},\\[3pt]K_{3}&=- \frac{\alpha + \beta}{2},\\[3pt]K_{5}&=\frac{\alpha}{2},\\[3pt]K_{6}&=- 2 \alpha \beta - 1.\end{align*}
The structure restrictions are $\eta=\pm1$.\endgroup
\clearpage\subsection{$G_4$; distribution $D_2$}
\begingroup\small The connection matrices are\begin{gather*}
A_1=\begin{pmatrix}0 & 0 & 0\\0 & - \frac{\sqrt{2}}{2} & \frac{\alpha}{2} + \frac{\beta}{2}\\0 & \beta - 2 \eta & \frac{\sqrt{2}}{2}\end{pmatrix}\\[3pt]A_2=\begin{pmatrix}- \frac{\sqrt{2}}{2} & 0 & 0\\0 & 0 & - \frac{\alpha}{2} + \frac{\beta}{2}\\0 & - \frac{\alpha}{2} + \frac{\beta}{2} & 0\end{pmatrix}\\[3pt]A_3=\begin{pmatrix}- \frac{\alpha}{2} + \frac{\beta}{2} & 0 & 0\\0 & 0 & - \frac{\sqrt{2}}{2}\\0 & - \frac{\sqrt{2}}{2} & 0\end{pmatrix}
\end{gather*}The six component equations are\[\left\{\begin{aligned}
2 \mu&=0,\\[3pt]- x - y + \frac{z \left(\alpha + \beta\right)}{2}&=K_{2},\\[3pt]- \frac{x \left(\alpha - \beta\right)}{2} + y \left(- \beta + 2 \eta\right) - \frac{z}{2}&=K_{3},\\[3pt]2 \mu + z \left(- \alpha + \beta\right)&=K_{4},\\[3pt]\frac{y \left(\alpha - \beta\right)}{2} - \frac{z}{2}&=K_{5},\\[3pt]- 2 \mu + 2 y&=K_{6}.\end{aligned}\right.\]
Here the symbols $K_j$ are local to this subsection:\begin{align*}
K_{2}&=\begin{aligned} &\alpha \eta + \frac{\beta^{2}}{2} - \beta \eta + \frac{1}{2}\\ &\quad - \frac{\alpha \beta}{2}\end{aligned},\\[3pt]K_{3}&=- \frac{\alpha}{2},\\[3pt]K_{4}&=\begin{aligned} &- \frac{\alpha^{2}}{2} - 2 \alpha \eta + \frac{\beta^{2}}{2} + 1\\ &\quad +2 \alpha \beta - 2 \beta \eta\end{aligned},\\[3pt]K_{5}&=\frac{\alpha + \beta}{2},\\[3pt]K_{6}&=- 2 \alpha \beta - 1.\end{align*}
The structure restrictions are $\eta=\pm1$.\endgroup
\clearpage\subsection{$G_5$; distribution $D_1$}
\begingroup\small The connection matrices are\begin{gather*}
A_1=\begin{pmatrix}0 & 0 & \frac{\alpha}{2} + \frac{\beta}{2} + \frac{\gamma}{2} + \frac{\delta}{2}\\0 & 0 & 0\\\frac{\alpha}{2} + \frac{\beta}{2} + \frac{\gamma}{2} + \frac{\delta}{2} & 0 & 0\end{pmatrix}\\[3pt]A_2=\begin{pmatrix}0 & 0 & \frac{\alpha}{2} + \frac{\beta}{2} - \frac{\gamma}{2} - \frac{\delta}{2}\\0 & 0 & 0\\0 & 0 & 0\end{pmatrix}\\[3pt]A_3=\begin{pmatrix}0 & 0 & 0\\0 & - \frac{\alpha}{2} + \frac{\beta}{2} + \frac{\gamma}{2} - \frac{\delta}{2} & 0\\0 & 0 & 0\end{pmatrix}
\end{gather*}The six component equations are\[\left\{\begin{aligned}
2 \mu + z \left(\alpha + \beta + \gamma + \delta\right)&=K_{1},\\[3pt]\frac{z \left(\alpha + \beta - \gamma - \delta\right)}{2}&=0,\\[3pt]- \frac{x \left(\alpha + \beta + \gamma + \delta\right)}{2}&=0,\\[3pt]2 \mu&=0,\\[3pt]- \frac{y \left(\alpha - \beta - \gamma + \delta\right)}{2}&=0,\\[3pt]- 2 \mu&=K_{6}.\end{aligned}\right.\]
Here the symbols $K_j$ are local to this subsection:\begin{align*}
K_{1}&=\begin{aligned} &- \frac{\alpha^{2}}{2} - \frac{\beta^{2}}{2} - \frac{\gamma^{2}}{2} - \frac{\delta^{2}}{2}\\ &\quad - \alpha \beta - \alpha \gamma - \alpha \delta - \beta \gamma\\ &\quad - \beta \delta - \gamma \delta\end{aligned},\\[3pt]K_{6}&=\begin{aligned} &\alpha^{2} + \alpha \beta + \alpha \gamma + \delta^{2}\\ &\quad +2 \beta \gamma + \beta \delta + \gamma \delta\end{aligned}.\end{align*}
The structure restrictions are $\alpha+\delta\ne0,\quad\alpha\gamma+\beta\delta=0$.\endgroup
\clearpage\subsection{$G_5$; distribution $D_2$}
\begingroup\small The connection matrices are\begin{gather*}
A_1=\begin{pmatrix}0 & 0 & 0\\0 & 0 & \frac{\alpha}{2} - \frac{\beta}{2} + \frac{\gamma}{2} - \frac{\delta}{2}\\0 & 0 & 0\end{pmatrix}\\[3pt]A_2=\begin{pmatrix}0 & 0 & 0\\0 & 0 & \frac{\alpha}{2} - \frac{\beta}{2} - \frac{\gamma}{2} + \frac{\delta}{2}\\0 & \frac{\alpha}{2} - \frac{\beta}{2} - \frac{\gamma}{2} + \frac{\delta}{2} & 0\end{pmatrix}\\[3pt]A_3=\begin{pmatrix}- \frac{\alpha}{2} - \frac{\beta}{2} - \frac{\gamma}{2} - \frac{\delta}{2} & 0 & 0\\0 & 0 & 0\\0 & 0 & 0\end{pmatrix}
\end{gather*}The six component equations are\[\left\{\begin{aligned}
2 \mu&=0,\\[3pt]\frac{z \left(\alpha - \beta + \gamma - \delta\right)}{2}&=0,\\[3pt]- \frac{x \left(\alpha + \beta + \gamma + \delta\right)}{2}&=0,\\[3pt]2 \mu + z \left(\alpha - \beta - \gamma + \delta\right)&=K_{4},\\[3pt]- \frac{y \left(\alpha - \beta - \gamma + \delta\right)}{2}&=0,\\[3pt]- 2 \mu&=K_{6}.\end{aligned}\right.\]
Here the symbols $K_j$ are local to this subsection:\begin{align*}
K_{4}&=\begin{aligned} &- \frac{\alpha^{2}}{2} - \frac{\beta^{2}}{2} - \frac{\gamma^{2}}{2} - \frac{\delta^{2}}{2}\\ &\quad +\alpha \beta + \alpha \gamma + \beta \delta + \gamma \delta\\ &\quad - \alpha \delta - \beta \gamma\end{aligned},\\[3pt]K_{6}&=\begin{aligned} &\alpha^{2} - \alpha \beta - \alpha \gamma + \delta^{2}\\ &\quad +2 \beta \gamma - \beta \delta - \gamma \delta\end{aligned}.\end{align*}
The structure restrictions are $\alpha+\delta\ne0,\quad\alpha\gamma+\beta\delta=0$.\endgroup
\clearpage\subsection{$G_6$; distribution $D_1$}
\begingroup\small The connection matrices are\begin{gather*}
A_1=\begin{pmatrix}0 & 0 & \frac{\gamma}{2}\\0 & \frac{\sqrt{2} \alpha}{2} & 0\\\frac{\gamma}{2} & 0 & 0\end{pmatrix}\\[3pt]A_2=\begin{pmatrix}\frac{\sqrt{2} \alpha}{2} & 0 & \frac{\gamma}{2}\\0 & 0 & 0\\\beta & 0 & \frac{\sqrt{2} \delta}{2}\end{pmatrix}\\[3pt]A_3=\begin{pmatrix}0 & 0 & - \frac{\sqrt{2} \delta}{2}\\0 & \frac{\gamma}{2} & 0\\- \frac{\sqrt{2} \delta}{2} & 0 & 0\end{pmatrix}
\end{gather*}The six component equations are\[\left\{\begin{aligned}
\gamma z + 2 \mu&=K_{1},\\[3pt]\alpha x + \alpha y + \frac{\gamma z}{2}&=K_{2},\\[3pt]- \frac{\gamma x}{2} - \frac{\delta z}{2}&=K_{3},\\[3pt]2 \mu&=0,\\[3pt]- \beta x + \frac{\gamma y}{2} - \frac{\delta z}{2}&=K_{5},\\[3pt]2 \delta x - 2 \mu&=K_{6}.\end{aligned}\right.\]
Here the symbols $K_j$ are local to this subsection:\begin{align*}
K_{1}&=\frac{2 \beta \gamma - \gamma^{2} + 2 \delta^{2}}{2},\\[3pt]K_{2}&=\frac{2 \alpha^{2} + \alpha \delta + \beta \gamma}{2},\\[3pt]K_{3}&=- \frac{\gamma \left(\alpha - \delta\right)}{4},\\[3pt]K_{5}&=\frac{\gamma \left(\alpha + \delta\right)}{2},\\[3pt]K_{6}&=- \delta^{2}.\end{align*}
The structure restrictions are $\alpha+\delta\ne0,\quad\alpha\gamma-\beta\delta=0$.\endgroup
\clearpage\subsection{$G_6$; distribution $D_2$}
\begingroup\small The connection matrices are\begin{gather*}
A_1=\begin{pmatrix}0 & 0 & 0\\0 & \frac{\sqrt{2} \alpha}{2} & - \frac{\gamma}{2}\\0 & - \beta & \frac{\sqrt{2} \delta}{2}\end{pmatrix}\\[3pt]A_2=\begin{pmatrix}\frac{\sqrt{2} \alpha}{2} & 0 & 0\\0 & 0 & - \frac{\gamma}{2}\\0 & - \frac{\gamma}{2} & 0\end{pmatrix}\\[3pt]A_3=\begin{pmatrix}- \frac{\gamma}{2} & 0 & 0\\0 & 0 & - \frac{\sqrt{2} \delta}{2}\\0 & - \frac{\sqrt{2} \delta}{2} & 0\end{pmatrix}
\end{gather*}The six component equations are\[\left\{\begin{aligned}
2 \mu&=0,\\[3pt]\alpha x + \alpha y - \frac{\gamma z}{2}&=K_{2},\\[3pt]\beta y - \frac{\gamma x}{2} - \frac{\delta z}{2}&=K_{3},\\[3pt]- \gamma z + 2 \mu&=K_{4},\\[3pt]\frac{\gamma y}{2} - \frac{\delta z}{2}&=K_{5},\\[3pt]2 \delta y - 2 \mu&=K_{6}.\end{aligned}\right.\]
Here the symbols $K_j$ are local to this subsection:\begin{align*}
K_{2}&=\frac{2 \alpha^{2} + \alpha \delta + \beta \gamma}{2},\\[3pt]K_{3}&=- \frac{\gamma \left(\alpha + \delta\right)}{2},\\[3pt]K_{4}&=\frac{2 \beta \gamma - \gamma^{2} + 2 \delta^{2}}{2},\\[3pt]K_{5}&=\frac{\gamma \left(\alpha - \delta\right)}{4},\\[3pt]K_{6}&=- \delta^{2}.\end{align*}
The structure restrictions are $\alpha+\delta\ne0,\quad\alpha\gamma-\beta\delta=0$.\endgroup
\clearpage\subsection{$G_7$; distribution $D_1$}
\begingroup\small The connection matrices are\begin{gather*}
A_1=\begin{pmatrix}0 & 0 & \frac{\alpha}{2} + \frac{\beta}{2} + \frac{\gamma}{2} + \frac{\delta}{2}\\0 & \frac{\sqrt{2} \left(\alpha - \beta\right)}{2} & 0\\\frac{\alpha}{2} + \frac{\beta}{2} + \frac{\gamma}{2} + \frac{\delta}{2} & 0 & 0\end{pmatrix}\\[3pt]A_2=\begin{pmatrix}\frac{\sqrt{2} \left(- \alpha - \beta\right)}{2} & 0 & \frac{\alpha}{2} + \frac{\beta}{2} - \frac{\gamma}{2} - \frac{\delta}{2}\\0 & 0 & 0\\- \beta & 0 & \frac{\sqrt{2} \left(\beta - \delta\right)}{2}\end{pmatrix}\\[3pt]A_3=\begin{pmatrix}0 & 0 & - \frac{\sqrt{2} \beta}{2} - \frac{\sqrt{2} \delta}{2}\\0 & - \frac{\alpha}{2} + \frac{\beta}{2} + \frac{\gamma}{2} - \frac{\delta}{2} & 0\\- \frac{\sqrt{2} \beta}{2} - \frac{\sqrt{2} \delta}{2} & 0 & 0\end{pmatrix}
\end{gather*}The six component equations are\[\left\{\begin{aligned}
2 \mu + z \left(\alpha + \beta + \gamma + \delta\right)&=K_{1},\\[3pt]x \left(- \alpha - \beta\right) + y \left(\alpha - \beta\right) + \frac{z \left(\alpha + \beta - \gamma - \delta\right)}{2}&=K_{2},\\[3pt]- \frac{x \left(\alpha + \beta + \gamma + \delta\right)}{2} - \frac{z \left(\beta + \delta\right)}{2}&=K_{3},\\[3pt]2 \mu&=0,\\[3pt]\beta x - \frac{y \left(\alpha - \beta - \gamma + \delta\right)}{2} - \frac{z \left(\beta - \delta\right)}{2}&=K_{5},\\[3pt]- 2 \mu + 2 x \left(\beta + \delta\right)&=K_{6}.\end{aligned}\right.\]
Here the symbols $K_j$ are local to this subsection:\begin{align*}
K_{1}&=\begin{aligned} &- \frac{\alpha^{2}}{2} - \frac{\beta^{2}}{2} - \frac{\gamma^{2}}{2} + \frac{\delta^{2}}{2}\\ &\quad - \alpha \gamma - \alpha \delta - \gamma \delta\end{aligned},\\[3pt]K_{2}&=\begin{aligned} &- \alpha^{2} + \alpha \beta - \frac{\alpha \delta}{2} + \beta \delta\\ &\quad - \frac{\beta \gamma}{2}\end{aligned},\\[3pt]K_{3}&=\begin{aligned} &- \frac{\alpha^{2}}{4} + \frac{\alpha \delta}{2} + \frac{\beta^{2}}{2} - \frac{\delta^{2}}{4}\\ &\quad - \frac{\alpha \beta}{4} - \frac{\alpha \gamma}{4} - \frac{\beta \gamma}{2} + \frac{\beta \delta}{4}\\ &\quad +\frac{\gamma \delta}{4}\end{aligned},\\[3pt]K_{5}&=\begin{aligned} &\frac{\alpha^{2}}{2} + \frac{\alpha \delta}{2} + \frac{\beta \gamma}{2} + \frac{\delta^{2}}{2}\\ &\quad - \frac{\alpha \beta}{2} - \frac{\alpha \gamma}{2} - \frac{\beta \delta}{2} - \frac{\gamma \delta}{2}\end{aligned},\\[3pt]K_{6}&=\begin{aligned} &\alpha^{2} + \alpha \beta + \alpha \gamma - \beta^{2}\\ &\quad +2 \beta \gamma - \beta \delta + \gamma \delta\end{aligned}.\end{align*}
The structure restrictions are $\alpha+\delta\ne0,\quad\alpha\gamma=0$.\endgroup
\clearpage\subsection{$G_7$; distribution $D_2$}
\begingroup\small The connection matrices are\begin{gather*}
A_1=\begin{pmatrix}0 & 0 & 0\\0 & \frac{\sqrt{2} \left(\alpha - \beta\right)}{2} & \frac{\alpha}{2} - \frac{\beta}{2} + \frac{\gamma}{2} - \frac{\delta}{2}\\0 & \beta & \frac{\sqrt{2} \left(\beta + \delta\right)}{2}\end{pmatrix}\\[3pt]A_2=\begin{pmatrix}\frac{\sqrt{2} \left(- \alpha - \beta\right)}{2} & 0 & 0\\0 & 0 & \frac{\alpha}{2} - \frac{\beta}{2} - \frac{\gamma}{2} + \frac{\delta}{2}\\0 & \frac{\alpha}{2} - \frac{\beta}{2} - \frac{\gamma}{2} + \frac{\delta}{2} & 0\end{pmatrix}\\[3pt]A_3=\begin{pmatrix}- \frac{\alpha}{2} - \frac{\beta}{2} - \frac{\gamma}{2} - \frac{\delta}{2} & 0 & 0\\0 & 0 & - \frac{\sqrt{2} \beta}{2} + \frac{\sqrt{2} \delta}{2}\\0 & - \frac{\sqrt{2} \beta}{2} + \frac{\sqrt{2} \delta}{2} & 0\end{pmatrix}
\end{gather*}The six component equations are\[\left\{\begin{aligned}
2 \mu&=0,\\[3pt]x \left(- \alpha - \beta\right) + y \left(\alpha - \beta\right) + \frac{z \left(\alpha - \beta + \gamma - \delta\right)}{2}&=K_{2},\\[3pt]- \beta y - \frac{x \left(\alpha + \beta + \gamma + \delta\right)}{2} - \frac{z \left(\beta + \delta\right)}{2}&=K_{3},\\[3pt]2 \mu + z \left(\alpha - \beta - \gamma + \delta\right)&=K_{4},\\[3pt]- \frac{y \left(\alpha - \beta - \gamma + \delta\right)}{2} - \frac{z \left(\beta - \delta\right)}{2}&=K_{5},\\[3pt]- 2 \mu + 2 y \left(\beta - \delta\right)&=K_{6}.\end{aligned}\right.\]
Here the symbols $K_j$ are local to this subsection:\begin{align*}
K_{2}&=\begin{aligned} &- \alpha^{2} - \alpha \beta - \frac{\alpha \delta}{2} - \beta \delta\\ &\quad - \frac{\beta \gamma}{2}\end{aligned},\\[3pt]K_{3}&=\begin{aligned} &- \frac{\alpha^{2}}{2} - \frac{\alpha \beta}{2} - \frac{\alpha \gamma}{2} - \frac{\delta^{2}}{2}\\ &\quad - \frac{\alpha \delta}{2} - \frac{\beta \gamma}{2} - \frac{\beta \delta}{2} - \frac{\gamma \delta}{2}\end{aligned},\\[3pt]K_{4}&=\begin{aligned} &- \frac{\alpha^{2}}{2} - \frac{\beta^{2}}{2} - \frac{\gamma^{2}}{2} + \frac{\delta^{2}}{2}\\ &\quad +\alpha \gamma - \alpha \delta + \gamma \delta\end{aligned},\\[3pt]K_{5}&=\begin{aligned} &\frac{\alpha^{2}}{4} - \frac{\beta^{2}}{2} + \frac{\beta \gamma}{2} + \frac{\delta^{2}}{4}\\ &\quad - \frac{\alpha \beta}{4} - \frac{\alpha \gamma}{4} - \frac{\alpha \delta}{2} + \frac{\beta \delta}{4}\\ &\quad +\frac{\gamma \delta}{4}\end{aligned},\\[3pt]K_{6}&=\begin{aligned} &\alpha^{2} - \alpha \beta - \beta^{2} + \beta \delta\\ &\quad - \alpha \gamma + 2 \beta \gamma - \gamma \delta\end{aligned}.\end{align*}
The structure restrictions are $\alpha+\delta\ne0,\quad\alpha\gamma=0$.\endgroup\clearpage

\section{Exact algebraic certificates and reproducibility}\label{app:cert}
\subsection{Sign and endpoint certificates}
The signs of the Sturm chains, in generation order, are listed below. The implementation first makes each defining polynomial monic, a positive rescaling that leaves all sign counts unchanged.
\begin{center}\small\begin{tabular}{ccc}\toprule
Polynomial&Signs at $-\infty$&Signs at $+\infty$\\\midrule
$P_-$ & $+,-,+,+,-,-,+$ & $+,+,+,-,-,+,+$ \\
$P_+$ & $+,-,+,+,-,-,+$ & $+,+,+,-,-,+,+$ \\
$Q_1$ & $+,-,+,-,+,-,-$ & $+,+,+,+,+,+,-$ \\
$Q_2$ & $+,-,+,-,-,+,+$ & $+,+,+,+,-,-,+$ \\
$H$ & $+,-,+,-,-$ & $+,+,+,+,-$ \\
\bottomrule\end{tabular}\end{center}
The following table gives reduced rational endpoints and exact variation counts. Each difference is one.
\begin{center}\begin{tabular}{cccc}\toprule
Polynomial&$(l,r)$&$\mathcal V(l)$&$\mathcal V(r)$\\\midrule
$P_-$ & $(-493/1000,-123/250)$ & 4 & 3 \\
$P_-$ & $(2737/500,219/40)$ & 3 & 2 \\
$P_+$ & $(-219/40,-2737/500)$ & 4 & 3 \\
$P_+$ & $(123/250,493/1000)$ & 3 & 2 \\
$Q_1$ & $(-439/1000,-219/500)$ & 5 & 4 \\
$Q_1$ & $(-73/250,-291/1000)$ & 4 & 3 \\
$Q_1$ & $(659/500,1319/1000)$ & 3 & 2 \\
$Q_1$ & $(3497/1000,1749/500)$ & 2 & 1 \\
$Q_2$ & $(-6599/1000,-3299/500)$ & 4 & 3 \\
$Q_2$ & $(719/1000,18/25)$ & 3 & 2 \\
$H$ & $(-2053/1000,-513/250)$ & 3 & 2 \\
$H$ & $(1447/1000,181/125)$ & 2 & 1 \\
\bottomrule\end{tabular}\end{center}
The sign-count principle follows from the remainder recursion. Away from zeros, all signs are locally constant. At a zero of an intermediate chain member, its neighbors have opposite signs, so the total variation count does not change. At a simple zero of the first member, the second member is its derivative and the variation count drops by one. The nonzero final constants certify that the defining polynomials and their derivatives are relatively prime. Endpoint differences therefore count every root exactly.

For clarity, after the corresponding normalizations and substitutions for $X,Y,Z$, the residual identities are
\[
\begin{array}{c|cc}
&F_{11}&F_{23}/\sqrt2\\\hline
G_2&f_2&g_2/8\\
G_4&f_4&-g_4/8.
\end{array}
\]
Thus the two-polynomial systems used in elimination are exactly equivalent to the two remaining tensor equations.

\subsection{Numerical identification of sextic branches}
The following values are rounded to six decimal places and are included solely to identify the exact algebraic branches. The $G_2$ table is normalized by $\gamma=1$; the $G_4$ table lists $(A,B,X,Y,Z)$ before restoring $\eta$. Rounding these values and substituting numerically will in general produce small nonzero residuals; all proofs use the exact polynomial representatives.

\begin{samepage}
\noindent\textbf{$G_2$ branches.}
\begin{center}
\begin{tabular}{crrrrr}
\toprule $\sigma$&$B$&$A$&$X$&$Y$&$Z$\\\midrule
$+1$&$-5.474309$&$-0.230093$&$1.759601$&$3.045398$&$3.523325$\\
$-1$&$-0.492365$&$3.644013$&$-1.294183$&$-0.656819$&$-2.201581$\\
$+1$&$0.492365$&$-3.644013$&$-1.294183$&$-0.656819$&$2.201581$\\
$-1$&$5.474309$&$0.230093$&$1.759601$&$3.045398$&$-3.523325$\\\bottomrule
\end{tabular}
\end{center}
\end{samepage}
\noindent\textbf{$G_4$ branches.}
\begin{center}\small
\begin{tabular}{crrrrr}
\toprule $j$&$B$&$A$&$X$&$Y$&$Z$\\\midrule
2&$-6.598819$&$-0.238867$&$-2.076242$&$-3.999627$&$6.367115$\\
1&$-0.438017$&$4.980325$&$1.681467$&$1.589187$&$-4.568452$\\
1&$-0.291489$&$0.459589$&$-0.366035$&$-1.031745$&$0.443020$\\
2&$0.719447$&$-0.774290$&$0.057061$&$0.400177$&$0.030392$\\
1&$1.318926$&$0.936046$&$-1.734576$&$-0.428686$&$1.590837$\\
1&$3.497863$&$0.209772$&$-1.233753$&$-1.081710$&$-0.349059$\\\bottomrule
\end{tabular}\end{center}
\subsection{Independent tensor reconstruction}
Let $T$ be the frame-change matrix in \eqref{eq:rotation}. The transformed constants are
\[
C_{ij}^k=\sum_{p,q,l}(T^{-1})_{kl}T_{pi}T_{qj}c_{pq}^l.
\]
The program below then uses the Koszul formula only in the prescribed blocks and projected brackets in mixed blocks. After computing $R_{ij}$, it contracts the unsymmetrized Ricci tensor. With $E=\diag(1,1,-1)$ and $N=(A_1\boldsymbol v,A_2\boldsymbol v,A_3\boldsymbol v)$, it forms
\[
F=N^TE+EN+\rho+\rho^T+2\mu E,\qquad
\boldsymbol v=(\sqrt2x,\sqrt2y,z)^T.
\]
This independent reconstruction prevents an assumed symmetric Ricci tensor or a copied component equation from hiding a sign error.

\subsection{Executable verification code}
The following script was run with SymPy 1.14.0 \cite{SymPy}. Save it as \texttt{verify\_classification.py} and run it in a Python environment with SymPy installed. It verifies tensor correspondence, resultants, denominator coprimality, lift identities, algebraic residuals, $G_7$ ideal equalities and all isolating intervals. The elementary sample substitutions supplement the exhaustive case proofs in the text; they do not replace those proofs. Computations use exact coefficients, and a failed assertion stops execution. The script writes \texttt{verification.txt} and \texttt{sturm.json}.
\begin{lstlisting}[language=Python]
import sys,json,pathlib
import sympy as S
sys.stdout.reconfigure(encoding='utf-8')
a,b,c,d,x,y,z,mu,eta=S.symbols('a b c d x y z mu eta', real=True)
rt=S.sqrt(2); E=S.diag(1,1,-1); T=S.Matrix([[1/rt,1/rt,0],[1/rt,-1/rt,0],[0,0,1]])
def bracket(i,j,cs):
 if i==j:return S.zeros(3,1)
 return S.Matrix(cs[(i,j)]) if i<j else -S.Matrix(cs[(j,i)])
def calc(cs,dist):
 cc={}
 for i,j in [(0,1),(0,2),(1,2)]:
  cc[i,j]=S.simplify(T.inv()*sum((T[p,i]*T[q,j]*bracket(p,q,cs) for p in range(3) for q in range(3)),S.zeros(3,1)))
 D={0,2} if dist==1 else {1,2}; A=[S.zeros(3) for i in range(3)]
 for i in range(3):
  for j in range(3):
   for k in range(3):
    if (i in D)==(j in D):
     if (k in D)==(j in D): A[i][k,j]=S.simplify((E[k,k]*bracket(i,j,cc)[k]-E[i,i]*bracket(j,k,cc)[i]+E[j,j]*bracket(k,i,cc)[j])/(2*E[k,k]))
    elif (k in D)==(j in D): A[i][k,j]=bracket(i,j,cc)[k]
 R={(i,j):S.simplify(A[i]*A[j]-A[j]*A[i]-sum((bracket(i,j,cc)[k]*A[k] for k in range(3)),S.zeros(3))) for i in range(3) for j in range(3)}
 rho=S.Matrix(3,3,lambda i,j:-sum(R[i,k][k,j] for k in range(3)))
 N=S.Matrix.hstack(*(m*S.Matrix([rt*x,rt*y,z]) for m in A))
 F=S.simplify(N.T*E+E*N+rho+rho.T+2*mu*E)
 return cc,A,rho,F
cs=[{(0,1):[a,0,-b],(0,2):[-a,-b,0],(1,2):[b,a,a]},
{(0,1):[0,c,-b],(0,2):[0,-b,-c],(1,2):[a,0,0]},
{(0,1):[0,0,-c],(0,2):[0,-b,0],(1,2):[a,0,0]},
{(0,1):[0,-1,2*eta-b],(0,2):[0,-b,1],(1,2):[a,0,0]},
{(0,1):[0,0,0],(0,2):[a,b,0],(1,2):[c,d,0]},
{(0,1):[0,a,b],(0,2):[0,c,d],(1,2):[0,0,0]},
{(0,1):[-a,-b,-b],(0,2):[a,b,b],(1,2):[c,d,d]}]
allF=[]; report=[]
def check(label,truth):
 assert truth,label
 report.append(label+': PASS'); print(report[-1],flush=True)
def zero(expr):return S.simplify(expr)==0
def reduced(expr,pol,var):return S.rem(S.together(expr).as_numer_denom()[0],pol,var)==0
for n,co in enumerate(cs,1):
 cc,A,rho,F=calc(co,1); _,AA,rr,FF=calc(co,2); allF.append(F.subs(mu,0))
 check('G%d steady diagonal'%n,F[1,1]==2*mu and FF[0,0]==2*mu)
 sub={x:y,y:x,z:-z}
 if n==1:sub[a]=-a
 if n in (5,7):sub.update({a:-a,d:-d})
 U=S.Matrix([[0,1,0],[1,0,0],[0,0,-1]])
 check('G%d D2 tensor equivalence'%n,all(zero(v) for v in FF-U*F.subs(sub,simultaneous=True)*U))
 print('G%d adapted brackets'%n,cc,flush=True)
f2=a**3*b-2*a*a*b*b+2*a*a+a*b**3-3*a*b-b*b-1
g2=2*a**4*b-2*a**3*b*b+3*a**3-2*a*a*b**3-a*a*b+2*a*b**4+9*a*b*b+10*a-3*b**3+6*b
P=lambda sig:16*b**6+64*sig*b**5-96*b**4+152*sig*b**3-144*b*b+88*sig*b-23
Ap=lambda sig:-(8*b**5+36*sig*b**4-28*b**3+70*sig*b*b-51*b+31*sig)/6
Q1=16*b**6-96*b**5+160*b**4-56*b**3-56*b*b+24*b+9
Q2=16*b**6+32*b**5-352*b**4+760*b**3-744*b*b+360*b-71
B1=-(16*b**5-96*b**4+172*b**3-116*b*b+10*b+15)/6
B2=(48*b**5+160*b**4-852*b**3+1116*b*b-562*b+81)/14
f4=f2+2*a+2*b;g4=g2+4*a*a-24*a*b+4*b*b-8
check('G2 resultant',S.expand(S.resultant(f2,g2,a)+b*P(1)*P(-1))==0)
check('G4 resultant sign',S.expand(S.resultant(f4,g4,a)+b*Q1*Q2)==0)
for label,f,g,pol,A0,idx in [('P+',f2,g2,P(1),Ap(1),1),('P-',f2,g2,P(-1),Ap(-1),1),('Q1',f4,g4,Q1,B1,3),('Q2',f4,g4,Q2,B2,3)]:
 X=a*b+S.Rational(1,2) if idx==1 else -a*b-S.Rational(1,2)
 Z=a+b+(a-b)*X if idx==1 else a+b-(a-b)*X
 Y=((a+b)*Z-a*b+b*b+1)/2-X if idx==1 else (a*b-2*a-b*b+2*b-(a+b)*Z-1)/2-X
 subs={x:X,y:Y,z:Z,c:1,eta:1}
 check(label+' all tensor residuals',all(reduced(v.subs(subs,simultaneous=True).subs(a,A0),pol,b) for v in allF[idx]))
 subres=S.subresultants(f,g,a)[-2]
 den=S.Poly(subres,a).coeff_monomial(a)
 check(label+' nonzero denominator',S.degree(S.gcd(den,pol),b)==0)
 print(label,'subresultant',S.factor(subres),flush=True)
 print(label,'Sturm',S.sturm(pol,b),flush=True)
 check(label+' squarefree',S.degree(S.gcd(pol,S.diff(pol,b)),b)==0)
F7=[S.expand(v) for v in [allF[6][0,0],allF[6][0,1],allF[6][0,2]/rt,allF[6][1,2]/rt,allF[6][2,2]]]
f71=[v.subs({a:1,c:0}) for v in F7]
base71=[2*x+1-b*d-d*d+d,2*y+2+3*d,2*z+2-b*d-d*d+d,b*b+2*b*d-2*b-d*d-2*d-1,d*(d*d+d+1)]
check('G7 alpha nonzero ideal equality',S.groebner(f71,x,y,z,b,d,domain=S.QQ)==S.groebner(base71,x,y,z,b,d,domain=S.QQ))
f70=[v.subs({a:0,d:1}) for v in F7]; H=c**4-3*c*c+2*c-1; J=c*c+2*c-1
G=S.groebner(f70,x,y,z,b,c,domain=S.QQ)
check('G7 elimination polynomial',G.reduce(c*c*J*H)[1]==0)
base0=[x+y-z+b*b/4+b/2+S.Rational(5,4),(b+1)*(y+1),b*b/2+(b+1)*z-S.Rational(1,2),(b+1)*(b*b+2*b-1)]
check('G7 C0 ideal equality',S.groebner([v.subs(c,0) for v in f70],x,y,z,b,domain=S.QQ)==S.groebner(base0,x,y,z,b,domain=S.QQ))
for label,pol,bas in [('quadratic',J,[x-c/2,y+1,z,b,J]),('quartic',H,[x+(c*c-1)/2,y+(c**3+c*c-2*c+3)/2,z+(c*c+c-1)/2,b+c**3-3*c+2,H])]:
 check('G7 '+label+' ideal equality',S.groebner(f70+[pol],x,y,z,b,c,domain=S.QQ)==S.groebner(bas,x,y,z,b,c,domain=S.QQ))
check('G3 residuals',all(zero(v.subs({a:0,b:0,x:0})) for v in allF[2]))
for sub in [{a:1,b:-1,c:0,d:0,y:0},{a:0,b:0,c:-1,d:1,y:0}]:check('G5 explicit residual',all(zero(v.subs(sub)) for v in allF[4]))
r0=(1-S.sqrt(5))/2
for sub in [{d:0,c:0,b:0,y:a-x},{d:0,c:0,x:0,y:a},{d:1,a:0,b:0,c:rt,x:-S.Rational(1,2),y:1,z:0},{d:1,a:r0,c:S.sqrt(-2*r0),b:r0*S.sqrt(-2*r0),x:-S.Rational(1,2),y:1+r0/2,z:r0*S.sqrt(-2*r0)/2}]:
 check('G6 explicit residual',all(zero(v.subs(sub,simultaneous=True)) for v in allF[5]))
intervals={'P-':[(-493,-492),(5474,5475)],'P+':[(-5475,-5474),(492,493)],'Q1':[(-439,-438),(-292,-291),(1318,1319),(3497,3498)],'Q2':[(-6599,-6598),(719,720)],'H':[(-2053,-2052),(1447,1448)]}
sturmdata=[]
for label,pol,var in [('P-',P(-1),b),('P+',P(1),b),('Q1',Q1,b),('Q2',Q2,b),('H',H,c)]:
 chain=S.sturm(pol,var)
 def signs_at(t):
  if t in [-S.oo,S.oo]:return [int(S.sign(S.LC(S.Poly(v,var)))*((-1)**S.degree(v,var) if t==-S.oo else 1)) for v in chain]
  return [int(S.sign(v.subs(var,t))) for v in chain]
 def variation(signs):
  s=[v for v in signs if v];return sum(s[i]!=s[i+1] for i in range(len(s)-1))
 signsminus=signs_at(-S.oo);signsplus=signs_at(S.oo)
 sturmdata.append({'label':label,'minus':signsminus,'plus':signsplus,'Vminus':variation(signsminus),'Vplus':variation(signsplus),'intervals':[]})
 for lo,hi in intervals[label]:
  l=S.Rational(lo,1000);h=S.Rational(hi,1000); vl=variation(signs_at(l));vh=variation(signs_at(h))
  check(label+' interval '+str((l,h)),vl-vh==1)
  sturmdata[-1]['intervals'].append([str(l),str(h),vl,vh])
 check(label+' real root total',variation(signsminus)-variation(signsplus)==len(intervals[label]))

for label,f,g,pol,A0 in [('P+',f2,g2,P(1),Ap(1)),('P-',f2,g2,P(-1),Ap(-1)),('Q1',f4,g4,Q1,B1),('Q2',f4,g4,Q2,B2)]:
 linear=S.subresultants(f,g,a)[-2]
 check(label+' lift identity',reduced(linear.subs(a,A0),pol,b))
check('G2 factors disjoint',S.degree(S.gcd(P(1),P(-1)),b)==0)
check('G4 factors disjoint',S.degree(S.gcd(Q1,Q2),b)==0)
check('G7 factors disjoint',S.gcd(c*J,H)==1 and J.subs(c,0)!=0)
check('H squarefree',S.gcd(H,S.diff(H,c))==1)
for n in [1,3]:
 X=a*b+S.Rational(1,2) if n==1 else -a*b-S.Rational(1,2)
 Z=a+b+(a-b)*X if n==1 else a+b-(a-b)*X
 Y=((a+b)*Z-a*b+b*b+1)/2-X if n==1 else (a*b-2*a-b*b+2*b-(a+b)*Z-1)/2-X
 rems=[S.factor(v.subs({x:X,y:Y,z:Z,c:1,eta:1},simultaneous=True)) for v in [allF[n][0,0],allF[n][1,2]/rt]]
 print('Residual pair G'+str(n+1),rems,flush=True)

pathlib.Path('verification.txt').write_text('\n'.join(report),encoding='utf-8')
pathlib.Path('sturm.json').write_text(json.dumps(sturmdata,indent=2),encoding='utf-8')
print('ALL CHECKS PASSED',flush=True)

\end{lstlisting}

\end{document}